\documentclass[letterpaper,12pt]{amsart}
\title[Monoidal Gr\"obner bases and applications]{Finite Gr\"obner bases in infinite dimensional \\ polynomial rings and applications}

\author{Christopher J. Hillar}
\address{Mathematical Sciences Research Institute, 17 Gauss Way, Berkeley, CA 94720} 
\email{chillar@msri.org}
\author{Seth Sullivant}
\address{Department of Mathematics  \\ North Carolina State University, Raleigh, NC}
\email{smsulli2@ncsu.edu}

\thanks{Hillar was partially supported by an NSA Young Investigator Grant and an NSF All-Institutes Postdoctoral Fellowship administered by the Mathematical Sciences Research Institute through its core grant DMS-0441170.  Sullivant is partially supported by  NSF grant DMS-0840795.  Part of this research was carried out during visits to SAMSI}

\date{\today}

\usepackage{latexsym,array,delarray,amsthm,amssymb,epsfig}
\usepackage{pstricks}

\theoremstyle{plain}
\newtheorem{thm}{Theorem}[section]
\newtheorem{lemma}[thm]{Lemma}
\newtheorem{prop}[thm]{Proposition}
\newtheorem{cor}[thm]{Corollary}

\theoremstyle{definition}
\newtheorem{defn}[thm]{Definition}
\newtheorem{ex}[thm]{Example}
\newtheorem{pr}[thm]{Problem}

\newtheorem{ques}[thm]{Question}
\newtheorem{rmk}[thm]{Remark}

\newcommand{\ind}{\mbox{$\perp \kern-5.5pt \perp$}}
\newcommand{\indsub}{{\mbox{\scriptsize$\perp \kern-4.5pt \perp$}}}

\DeclareMathSymbol{\N}{\mathbin}{AMSb}{"4E}
\DeclareMathSymbol{\Z}{\mathbin}{AMSb}{"5A}

\newcommand{\<}{\langle}
\renewcommand{\>}{\rangle}

\def \N { {\mathbb N} } 
\def \R { {\mathbb R} }

\def \Z { {\mathbb Z} }
\def \K { {\mathbb K} }
\def \P { {\mathbb P} }
\def \S { {\frak S} }

\begin{document}
\maketitle

\begin{abstract}
We introduce the theory of monoidal Gr\"obner bases, a concept which generalizes the familiar notion in a polynomial ring and allows for a description of Gr\"obner bases of ideals that are stable under the action of a monoid.  
The main motivation for developing this theory is to prove finiteness results in commutative algebra and applications.  A basic theorem of this type is that ideals in infinitely many indeterminates stable under the action of the symmetric group are finitely generated up to symmetry.  Using this machinery, we give new streamlined proofs of some classical finiteness theorems in algebraic statistics as well as a proof of the independent set conjecture of Ho\c{s}ten and the second author.
\end{abstract}


\section{Introduction}

In commutative algebra and its applications, one is frequently presented with a family of ideals in increasingly larger polynomial rings, and often it is observed that, up to some natural symmetry of the ideals, there exists a finite set of polynomials generating all of them.  Such situations arise in universal algebra and group theory \cite{Cohen67, Drensky06}, algebraic statistics \cite{Aoki2003, Santos2003, DrtonSturmSull07,Hosten2007, Draisma2008, BrouwerDraisma}, algebraic problems in chemistry \cite{Ruch3, Aschenbrenner2007, Draisma2008}, and in  classical results from combinatorial commutative algebra (for instance, that the $k \times k$
minors of a generic matrix form a Gr\"obner basis for the ideal they generate \cite{Sturmfels1990}).  The particular form of one of these finiteness results typically depends on the specifics of the family of ideals.  However, one wonders if there is a general principle at work that can explain a large portion of these phenomena.

We propose a general framework for proving finiteness theorems in rings with a monoid action.  In this setting, a finiteness theorem takes one of two forms:  (1) that a certain module over a noncommutative ring is Noetherian or (2) that a chain of ideals involving a monoidal filtration stabilizes.  Although the precise formulation of our theory requires the setup found in Section \ref{Pgroebner}, a typical result of the first type has the following flavor:

\begin{thm}\label{thm:sn}
The polynomial ring $\K[X_{[r] \times \P}]$ is a Noetherian $\K[X_{[r] \times \P}] \ast \S_\P$-module.
\end{thm}

Here, $\K[X_{[r] \times \P}]$ is a polynomial ring over a field $\K$ in the indeterminates $x_{i,j}$ with $i \in [r]:= \{1,2, \ldots, r \}$ and $j \in \P := \{1,2,3, \ldots \}$, the set of positive integers. Also, $\S_\P$ is the set of permutations of $\P$, acting on  $\K[X_{[r] \times \P}]$ by way of $\sigma \cdot x_{i,j} = x_{i, \sigma(j)}$, and the ring $\K[X_{[r] \times \P}] \ast \S_\P$ is the skew-monoid ring associated to $\K[X_{[r] \times \P}]$ and $\S_\P$ (see Section \ref{Pgroebner} for more details).  Stated simply, Theorem  \ref{thm:sn} says that every ideal in $\K[X_{[r] \times \P}]$ that is stable under the action of $\S_\P$ has a finite generating set up to $\S_\P$ symmetry.  

A version of Theorem \ref{thm:sn} was first proved by Cohen \cite{Cohen67} in an application to the theory of free metabelian groups, and then rediscovered much later in the study of some polynomial finiteness questions inspired by chemistry \cite{Aschenbrenner2007} (see also \cite{kemer2008} for another recent proof). Here, we study its application to algebraic statistics, and in particular its uniform treatment of some classical results in that field \cite{Hosten2007, Santos2003}.  Recent work by Draisma on finiteness problems for the factor analysis model \cite{Draisma2008} also depends on Theorem \ref{thm:sn}.

To prove Theorem \ref{thm:sn} and similar results, we shall develop a suitable theory of Gr\"obner bases for certain modules over (noncommutative) rings.  Section \ref{Pgroebner} contains this general theory of monoidal Gr\"obner bases and is the technical heart of the paper.  In this framework, we have a monoid $P$ of endomorphisms acting on a semigroup ring $\K[Q]$ (over a field $\K$), and a partial order (called the \textit{$P$-divisibility order}) on the monomials of $\K[Q]$ that respects this action.  Theorem \ref{main_GB_thm}, the main result in Section \ref{Pgroebner}, is then the statement that finite Gr\"obner bases exist with respect to the monoid $P$  if and only if $P$-divisibility is a well-partial-ordering.  In many cases of interest (such as in our applications to algebraic statistics), this order condition is straightforward to check, leading directly to finite generation of ideals up to $P$-action.  For instance, in the particular case of Theorem \ref{thm:sn}, the condition reduces to a classical lemma of Higman \cite{Higman} in the order theory of words.  Not surprisingly, all known proofs of Theorem \ref{thm:sn} use Higman's lemma in an essential way.

We also introduce in Section \ref{Pgroebner} the concept of a filtration for a chain of ideals subject to the action of the monoid $P$ (Definition \ref{def:filtration}).  This notion allows us to pass from ideals in finitely many variables to ideals in infinitely many variables, and it can be used to formulate and prove finiteness theorems for $P$-invariant chains of ideals.  Our main result in this regard is Theorem \ref{thm:symminvarstab}; it says that a $P$-invariant chain stabilizes with respect to a filtration (also) when $P$-divisibility is a well-partial-order.

Section \ref{examples} is concerned with the major implications of the theory contained in Section \ref{Pgroebner} and, in particular, a proof of Theorem \ref{thm:sn}.  Beyond this result, we also provide a strategy using quotient modules for proving finite generation theorems for special ideals in rings (such as $\K[X_{\P \times \P}]$) that are not Noetherian modules over skew-group rings (such as $\K[X_{\P \times \P}] \ast ( \S_\P \times \S_\P)$).  Section \ref{sec:algstat} contains our application of these ideas to finiteness theorems for Markov bases in algebraic statistics, including new proofs of the main results in \cite{Hosten2007,Santos2003} as well as a proof of the independent set conjecture \cite[Conj. 4.6]{Hosten2007}.  The latter result, stated as Theorem \ref{thm:independentset} below and proved using filtrations, gives a finiteness property for Markov bases in models that have independent vertex sets.  

Finally, Section \ref{sec:further} is devoted to a discussion of questions and problems left unresolved by this paper.  In particular, the computational consequences of our work remain open.  


\section{Monoidal Gr\"obner Bases}\label{Pgroebner}

In this section we develop our most important basic tools: finiteness theorems for invariant ideals of monoidal rings.  These ideas generalize those of Aschenbrenner and the first author \cite{Aschenbrenner2007}, and the proofs use similar ideas.  The importance of our generalization comes both from its usefulness, which will be illustrated throughout the paper, and from   our distillation and simplification of the main techniques from \cite{Aschenbrenner2007}, which might be of independent interest.  

The main results of this section are Theorem \ref{main_GB_thm} and Theorem \ref{thm:symminvarstab}.  Theorem \ref{main_GB_thm} gives a finiteness criterion for monoidal Gr\"obner bases which we will combine with Higman's lemma in Section \ref{examples} to prove Theorem \ref{thm:sn} from the introduction.  Our other main result, Theorem \ref{thm:symminvarstab}, gives the same criterion for chains of ideals to stabilize, and we will use it to prove the independent set conjecture \cite{Hosten2007} in algebraic statistics (Theorem \ref{thm:independentset} below).

We begin with an abstract setting.  Let $\K$ be a field, let $Q$ be a (possibly noncommutative) semigroup with identity (also called a \textit{monoid}), and let $\K[Q]$ be the semigroup ring associated to $Q$ (over $\K$).  We call the elements of $Q$ the \textit{monomials} of $\K[Q]$.  
Let $P$ be a monoid of $\K$-algebra endomorphisms of $\K[Q]$ (with multiplication in $P$ given by composition).    

Associated to $\K[Q]$ and $P$ is the \textit{skew-monoid ring} $\K[Q] \ast  P$, which is formally the set of all 
linear combinations,  
\[ \K[Q] \ast P = \left\{ \sum_{i=1}^k c_i q_i p_i \ : \  c_i \in \K, q_i \in Q, \ p_i \in P \right\}.\]
Multiplication of monomials in the ring $\K[Q] \ast  P$  is given by \[q_1 p_1 \cdot q_2 p_2 = q_1 (p_1 q_2) (p_1p_2),\] and extended by distributivity to the whole ring.  Note that $p_{1}q_{2}$ in this expression denotes the result of applying the endomorphism $p_1$ to $q_{2}$ which is in $\K[Q]$ but is not necessarily a monomial.  The natural (left) action of the skew-monoid ring on $\K[Q]$ makes $\K[Q]$ into a (left) module over $\K[Q] \ast  P$ as one can readily verify.\footnote{We must use the skew-monoid ring $\K[Q] \ast  P$ instead of the \textit{monoid ring} $\K[Q][P]$ to ensure that  $\K[Q]$ is a module.   The authors of \cite{Aschenbrenner2007} made such a mistake although none of the results there need to be modified except to make this adjustment (the ring structure of $\K[Q] \ast  P$ was not used in their proofs).} 

We say that a (left) ideal $I \subseteq \K[Q]$ is \textit{P-invariant} if 
\[PI := \{pn : p \in P,\ n \in I\} =  I.\]
Stated another way, a $P$-invariant ideal is simply a $\K[Q] \ast P$-submodule of $\K[Q]$. 
We want to provide a general setting for defining what it means for a $P$-invariant ideal $I$ of $\K[Q]$ to have a 
$P$-Gr\"obner basis.  Of specific interest for applications is when $I$ has a finite $P$-Gr\"obner basis, and our main contribution is a sufficient condition on $P$ and $Q$ under which this happens (see Theorem \ref{main_GB_thm}).  The examples found in the next section will illustrate the usefulness of our general framework.

\begin{rmk}
In many of our applications, $Q$ will be a subsemigroup of the semigroup of natural number sequences with finite support  (so that $\K[Q]$ is a subring of a polynomial ring), and $P$ will be defined using maps on the indices of the indeterminates in that polynomial ring.  When $P = \{\bold{1}\}$ consists of only the identity and $\K[Q]$ is a polynomial ring in a finite number of variables, we recover the classical formulation of Gr\"obner bases (see e.g.~\cite[Ch.~2]{Cox1}).  
\end{rmk}

If we have a total ordering $\preceq$ of $Q$, we can speak of the \textit{initial monomial} or \textit{leading monomial} $q = {\rm in}_\prec(f)$ of any nonzero $f \in \K[Q]$, which is the largest element $q \in Q$ with respect to $\preceq$ appearing with nonzero coefficient in $f$.  For notational convenience, we set ${\rm in}_\prec(f) = 0$ whenever $f = 0$, and also $0 \prec q$ for all $q \in Q$.  We are interested in those orderings which are naturally compatible with the linear action of $\K[Q] \ast P$.

\begin{defn}[$P$-orders]\label{defn:order}
A well-ordering $\preceq$ of $Q$ is called a $P$-\textit{order} on $\K[Q]$ if for all $q \in Q$, $p \in P$, and $f \in \K[Q]$, we have
$$ {\rm in}_\prec (qp \cdot f) = {\rm in}_\prec(qp \cdot {\rm in}_\prec(f)).$$
\end{defn}

In the next section, we shall provide examples of $P$-orders.  
The most important example of a $P$-order for us will be the \textit{shift order} on monomials (see Theorem \ref{thm:basicpoly}).   


Some basic facts about $P$-orders are collected in the following lemma.  Note that when $P = \{\bold{1}\}$, a $P$-order is simply a \textit{term order} on monomials.  For a useful characterization of $P$-orders, see Proposition \ref{prop:orderchar} below.

\begin{lemma}\label{lemma:Porderfacts}
Suppose that $\preceq$ is a $P$-order on $\K[Q]$.  Then the following hold:
\begin{enumerate}
\item For all $q\in Q$, $p \in P$, and $q_1,q_2 \in Q$, we have $q_1 \prec q_2 \  \Longrightarrow  \ {\rm in}_{\prec}(qpq_1) \preceq {\rm in}_{\prec}(qpq_2)$.  
\item If ${\rm in}_{\prec}(qpf) = {\rm in}_{\prec}(qpg)$ for some $q \in Q$, $p \in P$ and $f,g \in \K[Q]$, then either \mbox{${\rm in}_{\prec}(f) = {\rm in}_{\prec}(g)$} or $qpf = qpg = 0$.  
\item $Q$ is left-cancellative: for all $q,q_1,q_2 \in Q$, we have
$qq_1 = qq_2 \  \Longrightarrow \  q_1 = q_2.$
\item $q_2 \preceq q_1 q_2$ for all $q_1,q_2 \in Q$ (in particular, $1$ is the smallest monomial). 
\item All endomorphisms in $P$ are injective.
\item For all $q \in Q$ and $p \in P$, we have $q \preceq {\rm in}_\prec(pq)$.
\end{enumerate}
\end{lemma}
\begin{proof}
(1):   If ${\rm in}_{\prec}(qpq_1) \neq {\rm in}_{\prec}(qpq_2)$, then 
\[ \max\{{\rm in}_\prec (qpq_1), {\rm in}_\prec (qpq_2) \} = {\rm in}_\prec (qpq_1 + qp q_2) =  {\rm in}_\prec ( qp \cdot {\rm in}_\prec (q_1 + q_2)) ={\rm in}_\prec (qp q_2),\] and the claim follows. 

(2): If $ {\rm in}_{\prec}(qpg)= 0$, then $qpg =0$, so assume that ${\rm in}_{\prec}(qpf) = {\rm in}_{\prec}(qpg) \neq 0$.  If ${\rm in}_{\prec}(f) \prec {\rm in}_{\prec}(g)$, there exists $c \in \K$ such that  the leading terms of $qpg$ and $cqp f$ are the same.  This implies that,
\[    
{\rm in}_\prec ( qp \cdot {\rm in}_\prec (g)) = 
{\rm in}_\prec ( qp \cdot {\rm in}_\prec (g- cf)) =  
{\rm in}_\prec (qpg - cqp f)  \prec 
{\rm in}_\prec (qpg) =   
{\rm in}_\prec ( qp \cdot {\rm in}_\prec (g)),
\]
 which is a contradiction.    The first equality follows since ${\rm in}_{\prec}(f) \prec {\rm in}_{\prec}(g)$, the second since $\preceq$ is a $P$-order, the middle inequality since the leading terms of $qpg$ and $cqp f$ are the same, and the final equality follows again since $\preceq$ is a $P$ order.   Switching the roles of $f$ and $g$, we therefore have ${\rm in}_{\prec}(f) = {\rm in}_{\prec}(g)$.  
 
(3):  Follows directly from (2) with $p = \textbf{1}$, $f = q_1$, and $g =q_2$.

(4):  Suppose that $q_1 q_2 \preceq q_2$ for some $q_1,q_2 \in Q$.  Since $\preceq$ is a well-order, the infinite decreasing sequence obtained by using (1) repeatedly: \[ \cdots \preceq q_1^3 q_2 \preceq q_1^2 q_2 \preceq q_1 q_2 \preceq q_2,\] must terminate; in this case, we have $q_1^{k+1}q_2 = q_1^k q_2$ for some $k \in \N$.  It follows that $q_1 q_2 = q_2$ from (3), which proves (4).  

(5):  Let $p \in P$ and let $0 \neq f \in \K[Q]$.  From (1) and (4) and the fact that $p$ is a ring homomorphism, it follows that $1 = {\rm in}_{\prec}(p \cdot 1) \preceq {\rm in}_{\prec}(p \cdot {\rm in}_{\prec}(f)) = {\rm in}_{\prec}(pf)$.  Thus, $pf$ is nonzero for all $f \neq 0$, so $p$ is injective. 

(6):  Finally, suppose that ${\rm in}_\prec(pq) \preceq q$ for some $q \in Q$ and $p \in P$.  This gives us an infinite decreasing sequence, \[ \cdots \preceq {\rm in}_\prec(p^3 q) \preceq {\rm in}_\prec(p^2q) \preceq {\rm in}_\prec(pq) \preceq q.\]  Since $\preceq$ is a well-ordering, we must have ${\rm in}_\prec(p^{k+1} q) = {\rm in}_\prec(p^k q)$ for some $k \in \N$.  Using (2) and (5) in conjunction, it follows that ${\rm in}_\prec(pq) = q$, thereby proving (6).
\end{proof}

It turns out that properties (1) and (2) in Lemma \ref{lemma:Porderfacts} characterize when $P$-orders exist (the others follow from these).  As the following proposition demonstrates, we may further reduce the number of axioms to one.  This will be useful in proving that certain well-orderings on $Q$ are $P$-orders.

\begin{prop}[Characterization of $P$-orders]\label{prop:orderchar}
Let $Q$ be a monoid and let $P$ be a monoid of $\K$-algebra endomorphisms of $\K[Q]$.  Then a well-ordering $\preceq$ of $Q$ is a $P$-order if and only if for all $q \in Q$, $p \in P$, and $q_1,q_2 \in Q$, we have 
\[q_1 \prec q_2 \  \Longrightarrow  \ {\rm in}_{\prec}(qpq_1) \prec {\rm in}_{\prec}(qpq_2).\]
\end{prop}

\begin{proof}
Suppose first that $\preceq$ is a $P$-order.  By Lemma \ref{lemma:Porderfacts} part (1) we know that $q_{1} \prec q_{2}$ implies that ${\rm in}_{\prec}(qpq_1) \preceq {\rm in}_{\prec}(qpq_2)$.
   If ${\rm in}_{\prec}(qpq_1) = {\rm in}_{\prec}(qpq_2)$ for some $q \in Q$, $p \in P$, and $q_1,q_2 \in Q$, then Lemma \ref{lemma:Porderfacts} part (2) implies that $q_{1} = {\rm in}_{\prec}(q_{1}) = {\rm in}_{\prec}(q_{2}) = q_{2}$ or $qpq_1 = qpq_2 = 0$, and part (5) implies that the second option is not possible.  This proves the only-if direction. 

Conversely, suppose that $\preceq$ is a well-ordering of $Q$ satisfying the hypothesis of the proposition.  Let $q \in Q$, $p \in P$, and $0 \neq f \in \K[Q]$; we shall verify that ${\rm in}_{\prec}(qpf) = {\rm in}_{\prec}( qp \cdot {\rm in}_{\prec}(f))$.   Order the monomials $q_1 \prec \cdots \prec q_k$ appearing in $f$ with nonzero coefficient.  By assumption, we have ${\rm in}_{\prec}(qp q_{i})  \prec {\rm in}_{\prec}(qp q_{i+1})$ for all $i$.  It follows that ${\rm in}_{\prec}(qpf) = {\rm in}_{\prec}( qp \cdot {\rm in}_{\prec}(f))$ as desired.
\end{proof}

Having a $P$-order is quite restrictive as the following example demonstrates.

\begin{ex}[Semigroup ring without a $P$-order]

Let $\K[Q] = \K[X_\P]$ be the polynomial ring in infinitely many variables $X_{\P} = \{x_i: i \in \P \}$.  Also, let $P = \S_\P$ be the permutations of the positive integers $\P$, and let $\S_\P$ act on $\K[X_\P]$ by permuting indices.  Then there is no $\S_\P$-order on $\K[X_\P]$.  To see this, let $g = x_1 + x_2$, and suppose (without loss of generality) that a $P$-order makes ${\rm in}_\prec (g) = x_1$.  Then if $p = (12)$, we have  ${\rm in}_\prec(p \cdot g) = {\rm in}_\prec(g) = x_1$, while ${\rm in}_\prec(p \cdot {\rm in}_\prec (g)) = {\rm in}_\prec(p \cdot x_1) = x_2$.

More generally, if $R = \K[Q] \ast P$ where $P$ is a nontrivial group acting by permutations on $Q$, then there cannot exist a $P$-order on $\K[Q]$.  This will necessitate our study of special classes of monoids $P$.  \qed
\end{ex}

Before formulating a theory of Gr\"obner bases in this setting, we shall also need a relation (refining monomial divisibility) that is compatible with the canceling of leading monomials.  

\begin{defn}[The $P$-divisibility relation]\label{defn:Pdivis}
Given monomials $q_1,q_2 \in Q$, we say that $q_1  \, |_P \,   q_2$ if there exists $p \in P$ and $q \in Q$  such that $q_2 = q  \cdot {\rm in}_\prec(p q_1)$.  Such a $p$ is called a \textit{witness} for the relation $q_1  \, |_P \,  q_2$.
\end{defn}


\begin{prop}\label{lemma:Pdivispartialorder}
If $\preceq$ is a $P$-order on $Q$, then $P$-divisibility $ |_P $ is a partial order on $Q$ that is a coarsening of $\preceq$ (i.e., $q_1 \, |_P \,  q_2 \Longrightarrow q_1 \preceq q_2$).
\end{prop}
\begin{proof}
First of all, it is clear that $\, |_P \, $ is reflexive.  To prove transitivity, suppose that $q_2 = m_1 \cdot {\rm in}_\prec(p_1 q_1)$  and $q_3 = m_2 \cdot {\rm in}_\prec(p_2 q_2)$ for monomials $m_1,m_2 \in Q$ and $p_1,p_2 \in P$.  Using the fact that $p_2$ is a ring homomorphism and (repeatedly) the defining property of $P$-orders, we have,
\begin{equation*}
\begin{split}
 q_3 =  \  & m_2 \cdot {\rm in}_\prec(p_2 m_1 \cdot (p_2 \cdot {\rm in}_\prec(p_1 q_1))) \\
 = \ & m_2 \cdot {\rm in}_\prec(p_2 m_1 \cdot {\rm in}_\prec (p_2 \cdot  {\rm in}_\prec(p_1 q_1))) \\
 = \ & m_2 \cdot {\rm in}_\prec(p_2 m_1 \cdot {\rm in}_\prec (p_2p_1 q_1)). \\
\end{split}
\end{equation*}
Since ${\rm in}_\prec(p_2 m_1 \cdot {\rm in}_\prec (p_2p_1 q_1)) \neq 0$,  it must be of the form $q \cdot {\rm in}_\prec (p_2p_1 q_1)$ for some $q \in Q$.  It follows that $q_1\, |_P \,q_3$ with witness $p = p_2 p_1$.

Finally, to prove antisymmetry, it is enough to verify that $P$-divisibility is a coarsening of $\preceq$.   If $q_1 \, |_P \,  q_2$, then for some $p \in P$ and $q \in Q$, we have $q_2 = q \cdot  {\rm in}_\prec(pq_1)$.  Thus, by  properties (4) and (6) in Lemma \ref{lemma:Porderfacts}, we have $q_1 \preceq  {\rm in}_\prec(p q_1) \preceq q  \cdot {\rm in}_\prec(pq_1) = q_2$ as desired.
\end{proof}

If $\preceq$ is a $P$-order, then we may compute the \textit{initial final segment} with respect to the $P$-divisibility partial order of any subset $G \subseteq \K[Q]$:
\[{\rm in}_\prec(G) :=  \left \{  q : {\rm in}_\prec (g) \, |_P \, q  \text{ for some } g \in G \setminus \{0 \} \right \}.\]
It is clear that the set ${\rm in}_\prec(G)$ contains all the initial monomials of $G$. Moreover, when $I \subseteq \K[Q]$ is a $P$-invariant ideal, it is straightforward to check that it contains no other ones: \[ {\rm in}_\prec(I) =  \left \{  {\rm in}_\prec (f) :  f \in I \setminus \{0\}  \right \}.\]

\begin{rmk}
The schizophrenic terminology \emph{initial final segment} comes from the combination of two mathematical traditions.  From order theory, we have an upward closed subset of a partially ordered set, which is a final segment.  On the other hand, we have constructed this set by taking initial or leading terms of polynomials.

Note that the initial final segment is not an ideal (or initial segment) in the sense of order theory (as it is not closed downward).  Furthermore, it cannot, in general,  be made into a monomial ideal of $\K[Q]$, as is typically done in commutative algebra, because $P$ does not necessarily act by maps that send $Q$ to itself.
\end{rmk}

We now arrive at our definition of Gr\"obner bases for invariant ideals with respect to a given $P$-order.  We remark that a similar definition appears in \cite{BrouwerDraisma}, where they are given the name \textit{equivariant Gr\"obner bases}, and \cite{Drensky06} contains related work
in the noncommutative case (but without the assumption that the term order is compatible with the monoid actions).  

\begin{defn}\label{GBdef}
A set $G \subseteq I \subseteq \K[Q]$ is a $P$-\textit{Gr\"obner basis} for a $P$-invariant ideal $I$ (with respect to the $P$-order $\preceq$) if and only if
$$ {\rm in}_\prec(I) =  {\rm in}_\prec(G).$$
\end{defn}

Of course, the set $I$ can itself be considered a Gr\"obner basis for the ideal $I$, so the interest theoretically and computationally is when a finite Gr\"obner basis exists.  One goal of this section is to arrive at a criterion for $\preceq$ guaranteeing that finite $P$-Gr\"obner bases exist for all $P$-invariant $I$. 

In analogy with the classical case, a $P$-Gr\"obner basis generates the ideal up to the action of $P$.  Here, for an $R$-module $M$ and a subset $G \subseteq M$, the submodule $\<G\>_{R} \subseteq M$ is the $R$-module generated by $G$.

\begin{prop}\label{prop:GBgenerate}
If $G$ is a $P$-Gr\"obner basis for a $P$-invariant ideal $I \subseteq \K[Q]$, then \[ I = \<G\>_{\K[Q] \ast P}.\]
\end{prop}
\begin{proof}
Since $I$ is $P$-invariant, we have $\<G\>_{\K[Q] \ast P} \subseteq I$.  Conversely, given $f_1 \in I$, we shall prove $f \in \<G\>_{\K[Q] \ast P}$.  Since ${\rm in}_\prec(f_1) \in {\rm in}_\prec(I) = {\rm in}_\prec(G)$, there exist $q_1 \in Q$, $p_1 \in P$, and $g_1 \in G$ such that ${\rm in}_\prec(f_1) =   {\rm in}_\prec(q_1 p_1 g_1)$.  Thus, for some $c_1 \in \K$, the element \[f_2 := f_1 - c_1 q_1 p_1 g_1\] is either zero or has a smaller initial monomial than ${\rm in}_\prec(f_1)$.  Also, $f_2 \in I$, so there are $q_2 \in Q$, $p_2 \in Q$, and $g_2 \in G$ such that ${\rm in}_\prec(f_2) =   {\rm in}_\prec(q_2 p_2g_2)$.  As before, we define a new polynomial $f_3 := f_2 - c_2 q_2 p_2g_2$, which again is zero or has a smaller initial term.  Continuing in this way, we produce a sequence of polynomials $f_1, f_2, f_3, \ldots \in I$ all of whose initial terms form an infinite decreasing sequence.  Since $\preceq$ is a well-order, this sequence must terminate in a finite number of steps with some $f_{k+1} = 0$.  But then we have that $f_1 = \sum_{i =1}^k c_i q_i p_i g_i$ with the $g_i \in G$.   This proves the proposition.\end{proof}

If $P$-divisibility $\, |_P \,$ generates enough relations between elements of $Q$, then finite Gr\"obner bases for $P$-invariant ideals always exist. To state this result precisely, however, we need to introduce some basic definitions from order theory.  

Recall that a \textit{well-partial-ordering} $\leq$ on a set $S$ is a partial order such that (1) there are no infinite collections of pairwise incomparable elements (i.e., \emph{antichains}) and (2) there are no infinite strictly decreasing sequences.  This definition is a natural generalization of the notion of ``well-ordering" when $\leq$ is not total.  A \textit{final segment} is a subset $F \subseteq S$ which is closed upwards: $s \leq t \text{ and } s \in F \Rightarrow t \in F$ for all $s,t \in S$.  Given a subset $B \subseteq S$, the set \[\mathcal{F} (B) := \bigl\{t \in S :  b \leq t \text{ for some } b \in B  \bigr\}\] is a final segment of $S$, the {\it final segment generated by $B$.}  For example, with $P$-divisibility $\, |_P \,$ as the partial order, the set of monomials ${\rm in}_\prec(G)$ is a final segment generated by the initial monomials of $G$.  Thus, another way to state Definition \ref{GBdef} is to say that a subset $G \subseteq I$ is a $P$-Gr\"obner basis of $I$ if and only if the final segment generated by the leading monomials of $G$ contains all the leading monomials of $I$.  

Continuing further with order terminology, let us call an infinite sequence $s_1,s_2,\dots$ in $S$ {\it good}\/ if $s_i \leq s_j$ for some indices $i<j$, and {\it bad} otherwise.  The following elementary characterization of well-partial-orderings is classical \cite{Kruskal}.

\begin{prop}\label{prop:equivorder}
The following are equivalent for a partial order $\leq$ on a set $S$:
\begin{enumerate}
\item $S$ is well-partially-ordered.
\item Every infinite sequence in $S$ is good.
\item Every infinite sequence in $S$ contains an infinite increasing
  subsequence.
\item Any final segment of $S$ is finitely generated.
\item The ascending chain condition holds for final segments of $S$.
\end{enumerate}
\end{prop}

We now have all the ingredients to prove that finite $P$-Gr\"obner bases exist when $P$-divisibility is a well-partial-ordering (our finiteness criterion).  In the case that $Q = \N^k$, $P = \{\textbf{1}\}$, and $\preceq$ is any term order on $Q$, the theorem says that a finite Gr\"obner basis exists if monomial divisibility is a well-partial-order.  As this is the basic content of Dickson's Lemma, we recover the classical finiteness result for Gr\"obner bases in polynomial rings with a finite number of variables.  

\begin{thm}\label{main_GB_thm}
Let $\preceq$ be a $P$-order.  If $P$-divisibility $\, |_P \,$  is a well-partial-ordering, then every $P$-invariant ideal $I \subseteq \K[Q]$ has a finite $P$-Gr\"obner basis with respect to $\preceq$.  Moreover, if elements of $P$ send monomials to scalar multiples of monomials, the converse holds.
\end{thm}
\begin{proof}
The set of monomials ${\rm in}_{\prec}(I)$ is a final segment with respect to $P$-divisibility; thus, it is finitely generated by Proposition \ref{prop:equivorder}.  These generators are initial monomials of a finite subset $G$ of elements of $I$.  It follows that $G$ is a $P$-Gr\"obner basis.

For the second statement, we verify that (4) holds in the characterization of Proposition \ref{prop:equivorder}.   Let $M$ be any final segment of $Q$ with respect to $\, |_P \,$, and set $I = \<M \>_{\K[Q] \ast P}$.  By assumption, there is a finite set $G = \{g_1,\ldots,g_k\} \subseteq I$ such that  
\[
 M \subseteq {\rm in}_{\prec}( I) = 
 {\rm in}_{\prec}(G) = 
 \mathcal{F} \left( \{{\rm in}_{\prec}(g_1),\ldots,{\rm in}_{\prec}(g_k) \} \right).
 \]
Now, each $g \in G$ has a representation of the form \[ g = \sum_{j = 1}^d c_j q_j p_j m_j, \ \ \ c_j \in \K, \ q_j \in Q, \ p_j \in P, \ m_j \in M,\] and since elements of $P$ send monomials to scalar multiples of monomials, it follows that ${\rm in}_{\prec}(g) = q \cdot {\rm in}_{\prec}(pm)$ for some $q \in Q$, $p \in P$, and $m \in M$.  In particular, we have $m \, |_P \, {\rm in}_{\prec}(g)$.  Thus, $\mathcal{F} \left( \{{\rm in}_{\prec}(g_1),\ldots,{\rm in}_{\prec}(g_k) \} \right) \subseteq M$ and $M$ is finitely generated.
\end{proof}

\begin{rmk}
Define a \textit{monomial map} to be an element $p \in P$ that sends monomials to scalar multiples of monomials.  Theorem \ref{main_GB_thm} says that for a monoid $P$ of monomial maps, $P$-divisibility is a well-partial-ordering if and only if every $P$-invariant ideal has a finite $P$-Gr\"obner basis.   In our applications, the monoids $P$ consist entirely of monomial maps.  However, we do not know if the converse to Theorem \ref{main_GB_thm}  continues to hold when $P$ is a more general set of maps, and it would be interesting to understand this situation better.
\end{rmk}

Using Proposition \ref{prop:GBgenerate}, the following finiteness result is immediate. 

\begin{cor}\label{cor:GBfinitecor}
Let $\preceq$ be a $P$-order.  If $P$-divisibility $\, |_P \,$  is a well-partial-ordering, then every $P$-invariant ideal $I \subseteq \K[Q]$ is finitely generated over $\K[Q] \ast P$.  In other words, $\K[Q]$ is a Noetherian  $\K[Q] \ast P$-module.
\end{cor}

We next define a  general setup that allows us to go from global generation  to local stabilization (Theorem \ref{thm:symminvarstab}).  This can be seen as an analogue to \cite[Theorem 4.7]{Aschenbrenner2007} which guaranteed stabilization of certain $\frak S_{\P}$-invariant chains over a polynomial ring in an infinite number of indeterminates.  In fact, we shall show in the next section how the stabilization result of \cite{Aschenbrenner2007} follows from our theory.


\begin{defn}[Filtrations]\label{def:filtration}
Let $\preceq$ be a $P$-order, and suppose that $Q_n \subseteq Q$ and $P_{n,m} \subseteq P$ for nonnegative integers $m \geq n$.  We say that  $Q_n$ and $P_{n,m}$ is a \textit{filtration} of $\K[Q] \ast P$ if
\begin{enumerate}
\item Each $Q_n$ is a submonoid of $Q$.
\item $Q_n \subseteq Q_{n+1}$ for all $n$.  
\item $Q = \bigcup_{n}^{\infty} Q_n$ and  $P = \bigcup_{n,m =1}^{\infty} P_{n,m}$.
\item $P_{n,m} Q_n \subseteq \K[Q_m]$ for all $m \geq n$.
\item Each $P_{n,m}$ contains the identity endomorphism.
\item If $q \in Q_n \setminus Q_{n-1}$ and ${\rm in}_{\prec}(p q)  \in Q_m$ for some $p \in P$, then there exists $p' \in P_{n,m}$ with ${\rm in}_{\prec}(p' q) = {\rm in}_{\prec}(p q)$.
\item Each $Q_n$ is an \textit{initial segment} with respect to $\preceq$ (i.e., $u \preceq v$ and $v \in Q_n$ $\Rightarrow u \in Q_n$).  
\end{enumerate}
\end{defn}

\begin{rmk}\label{rmk:filtration}
From Lemma \ref{lemma:Porderfacts}, we have $q_1 \preceq q_1 q_2$ and $q_2 \preceq q_1 q_2$ for any $q_1,q_2 \in Q$.  In particular, if $q_1q_2 \in Q_n$, then (7) above implies that both  $q_1,q_2 \in Q_n$.
\end{rmk}

Our most important example of a filtration arises from decomposing the monoid of increasing functions.  It appears explicitly in the statements of Theorem \ref{thm:pifinite} and Corollary \ref{cor:snfinite}, and will be used to prove the independent set conjecture of  \cite[Conj. 4.6]{Hosten2007} (Theorem \ref{thm:independentset}).

Given a filtration of $\K[Q] \ast P$, we are interested in increasing chains $I_\circ$ of ideals $I_n \subseteq \K[Q_n]$: \[ I_{\circ} :=  I_1 \subseteq I_2 \subseteq \cdots \subseteq I_n \subseteq \cdots,\] simply called {\em chains}\/ below.  Of  primary importance is when these ideals stabilize ``up to the action'' of the monoid $P$.  For the purposes of this work, we will only consider a special class of chains; namely, a \textit{$P$-invariant chain} is one for which $P_{n,m} I_n \subseteq I_{m}$ for all $m \geq n$.  The stabilization definition alluded to above is as follows.

\begin{defn}\label{stabdef}
The $P$-invariant chain $I_\circ$ \textit{stabilizes} if there exists a positive integer $n_0$ such that \[ I_n = \bigcup_{k\leq n_0} \< P_{k,n} I_{k} \>_{\K[Q_n]}  \qquad\text{for all $n \geq n_0$.}\]
\end{defn}

In other words, a $P$-invariant chain $I_{\circ}$ stabilizes when the ideals $I_n$ can be generated by ``lifting" the finite set of ideals $\{I_1,\ldots,I_{n_0}\}$ in the chain. 

Any $P$-invariant chain $I_\circ$ naturally gives rise to an ideal $\mathcal{N}(I_\circ)$ over $\K[Q] \ast P$ by way of \[ \mathcal{N}(I_\circ) := \bigcup_{n \geq 1} I_n.\] 
It is easy to see that if $I_\circ$ stabilizes, then any set of $\K[Q_{n_0}]$-generators for $I_{n_0}$ will form a generating set of the $\K[Q] \ast P$-module $I = \mathcal{N}(I_\circ)$.  Our next result says that one can also move from global generation to chain stabilization; it will be a consequence of the following technical fact.

\begin{lemma}\label{lemma:GBfiltration}
Let $\preceq$ be a $P$-order and fix a filtration of $\K[Q] \ast P$.  Suppose that $G \subseteq \K[Q_{n_0}]$ is a finite $P$-Gr\"obner basis for a $P$-invariant ideal $I \subseteq \K[Q]$.  Then, if $0 \neq f \in \K[Q_n] \cap I$ with $n \geq n_0$, we have,
\[ {\rm in}_\prec(f) =  {\rm in}_\prec(qpg) \ \ \  \text{for some } \ q \in Q_n, \ p \in P_{k,n}, \ g \in G \cap \K[Q_k], \ k \leq n_0.\]
\end{lemma}
\begin{proof}
Let $0 \neq f \in  \K[Q_n] \cap I$.  Since $G$ is a $P$-Gr\"obner basis, it follows that \[\text{{\rm in}}_\prec(f) = q  \cdot {\rm in}_\prec(pg) \]
for some $q \in Q$, $p \in P$, and $g \in G$.  Since ${\rm in}_\prec(f) \in Q_n$, Remark \ref{rmk:filtration} implies that $q \in Q_n$ and ${\rm in}_\prec(pg) \in Q_n$.  Let $k \leq n_0$ be such that $ \text{{\rm in}}_\prec(g) \in Q_k \setminus Q_{k-1}$.  From the exchange property (6) of Definition \ref{def:filtration}, there is a $p' \in P_{k,n}$ such that  ${\rm in}_\prec(p' g) = {\rm in}_\prec(pg)$.  Thus, ${\rm in}_\prec(f) = q \cdot {\rm in}_\prec(p'g) = {\rm in}_\prec(qp'g)$.  Finally, since $Q_k$ is an initial segment, it follows that $g \in \K[Q_k]$.
\end{proof}

\begin{thm}[Chain stabilization]\label{thm:symminvarstab}
Let $\preceq$ be a $P$-order.  If $P$-divisibility $\, |_P \,$ is a well-partial-ordering, then every $P$-invariant chain stabilizes.
\end{thm}

\begin{proof}
Given an invariant chain $I_{\circ}$, construct the $P$-invariant ideal $I= \mathcal{N}(I_\circ)$ of $\K[Q]$, and let $G$ be a finite $P$-Gr\"obner basis for $I$ by Theorem \ref{main_GB_thm}.  The result now follows from Lemma \ref{lemma:GBfiltration} using a descent argument as in Proposition \ref{prop:GBgenerate}.
\end{proof}

\section{Examples, Counterexamples, and First Applications}\label{examples}

In this section, we begin to apply the abstract theory from Section \ref{Pgroebner} to specific examples that make frequent appearances in applications.  Although the finite Gr\"obner basis results we derive are for ideals invariant under the monoid of increasing functions, we can easily produce corollaries for the more familiar setting of ideals that are stable under a symmetric group action.  In Section \ref{sec:algstat}, we apply these ideas to Markov bases and other implicitization problems in algebraic statistics.

Our main monoid $P$ of interest for constructing monoidal Gr\"obner bases will be the monoid of increasing functions (the \textit{shift monoid}):
$$\Pi :=  \left\{ \pi :  \P \to \P :  \pi(i) < \pi(i+1) \mbox{ for all } i \in \P \right\}.$$

Given a set $R$, let $X_R =  \{x_r: r \in R \}$ denote the set of indeterminates indexed by $R$, and let $\K[X_R]$ be the (commutative) polynomial ring with coefficients in $\K$ and indeterminates $X_R$.  Of special interest is when $R$ is a product of the form $R = R_1 \times \cdots \times R_m$.  For $r \in \P$, let $[r] = \{1,2, \ldots, r \}$.  Our first result concerns the case $R= [r] \times \P$ with the (linear) action of $\Pi$ on  $\K[X_{[r] \times \P}]$ being generated by its action on the second index of the indeterminates $X_{[r] \times \P}$:  $$\pi x_{i,j}  :=  x_{i, \pi(j)}, \ \ \pi \in \Pi.$$

\begin{thm}\label{thm:basicpoly}
The column-wise lexicographic term order $x_{i,j} \preceq x_{k,l}$ if $j < l$ or ($j = l$ and $i \leq k$) is a $\Pi$-order on $\K[X_{[r] \times \P}]$ such that $\Pi$-invariant ideals of $\K[X_{[r] \times \P}]$ have finite $\Pi$-Gr\"obner bases.  In particular, the ring $\K[X_{[r] \times \P}]$ is a Noetherian $\K[X_{[r] \times \P}] \ast \Pi$-module.
\end{thm}

We call the $\Pi$-divisibility order induced by the column-wise lexicographic order in the statement of Theorem \ref{thm:basicpoly} the \textit{shift order}.  We shall prove Theorem \ref{thm:basicpoly} using Theorem \ref{main_GB_thm} by showing that the $\Pi$-divisibility partial order on monomials in $\K[X_{[r] \times \P}]$ is a well-partial-order.  Before verifying this fact,  we recall the notion of a Higman partial order associated to a well-partial-order.

\begin{defn}[The Higman Partial Order]\label{HigmanPO}
Let $(S, \preceq)$ be a partially-ordered set.  Let $(S_H, \preceq_H)$ be defined on the set $S_H = S^*$ of finite words of elements of $S$ by:
$$u_1u_2 \cdots u_n  \preceq_H v_1 v_2 \cdots v_m$$ 
if and only if there is a $\pi \in \Pi$ such that $u_i \preceq v_{\pi(i)}$ for $i \in [n]$.  
\end{defn}

The main result about Higman partial orders is Higman's Lemma \cite{Higman, NW1}.

\begin{lemma}[Higman's Lemma]\label{lem:hig} 
If $(S, \preceq)$ is a well-partial-order, then the Higman partial order $(S_H, \preceq_H)$ is also a well-partial-order.
\end{lemma}

Below, we shall apply Higman's Lemma to the set $S = \N^r$, partially ordered by inequality: 
\begin{equation*}\label{part_order}
(s_1,\ldots,s_r) \preceq (t_1,\ldots,t_r) :\Leftrightarrow s_i \leq t_i \text{ for } i = 1,\ldots,r.
\end{equation*}
This is a well-partial-order by Dickson's Lemma, and it can be interpreted as a well-partial-ordering on the monomials of $\K[X_{[r] \times \P}]$.

\begin{ex}
In the Higman ordering on words $(\N^2)^*$ induced by the partial order above, 
\[  (1,2)(3,1)(2,5) \preceq_H (1,0)(1,4)(5,2)(1,2)(2,7),\]
witnessed by any shift monoid element $\pi \in \Pi$ that has $\pi(1) = 2$, $\pi(2) = 3$, $\pi(3) = 5$.  Interpreted as a
$\Pi$-divisibility relation between monomials in the polynomial ring $\K[X_{[2] \times \P}]$, this statement reads:
\[  x_{1,1} x_{1,2}x_{2,2}^4 x_{1,3}^5x_{2,3}^2  x_{1,4}x_{2,4}^2 x_{1,5}^2x_{2,5}^7 = x_{1,1} x_{2,2}^2 x_{1,3}^2x_{2,3}  x_{1,4}x_{2,4}^2 x_{2,5}^2  \cdot \pi(  x_{1,1}x_{2,1}^2 x_{1,2}^3x_{2,2} x_{1,3}^2x_{2,3}^5).\] 
\qed
\end{ex}

\begin{proof}[Proof of Theorem \ref{thm:basicpoly}]
We first show that the column-wise lexicographic order is a $\Pi$-order on 
$\K[X_{[r] \times \P}]$.  Each monomial in $\K[X_{[r] \times \P}]$ has the 
form $x^u =  x_1^{u_1} \cdots x_n^{u_n}$ for some $n \in \P$, where 
$x_j^{u_j}  =  \prod_{i \in [r]} x_{i,j}^{u_{i,j}}$.  
Suppose that $x^u \prec x^v$.  Then we can write
$$x^u  =    x_1^{u_1} \cdots x_k^{u_k} x_{k+1}^{v_{k+1}} \cdots x_m^{v_n}$$ 
for some $k \leq n$ in which $x_k^{u_k} \prec x_k^{v_k}$.  For $\pi \in \Pi$, we have 
$$\pi x^u =   x_{\pi(1)}^{u_1} \cdots x_{\pi(k)}^{u_k} x_{\pi(k+1)}^{v_{k+1}} \cdots x_{\pi(n)}^{v_n},$$
$$\pi x^v =  x_{\pi(1)}^{v_1} \cdots x_{\pi(k)}^{v_k} x_{\pi(k+1)}^{v_{k+1}} \cdots x_{\pi(n)}^{v_n}.$$
Since $\pi$ is increasing, the right-most column index where $\pi x^u$ and $\pi x^v$ disagree is at $\pi(k)$, in which case $x_{\pi(k)}^{u_k} \prec x_{\pi(k)}^{v_k}$ so that $\pi x^u \prec \pi x^v$.  Since multiplication by an ordinary monomial preserves $\preceq$ for any term order, this proves that $\preceq$ is a $\Pi$-order by Proposition \ref{prop:orderchar}.

Next, we must show that $\Pi$-divisibility on $\K[X_{[r] \times \P}]$ is a well-partial-order.  In the $\Pi$-divisibility partial order, we have $x^u \, |_\Pi \, x^v$ if and only if there is a $\pi \in \Pi$ such that $ \pi x^u \, | \, x^v$ (monomial division-wise).  In turn, this happens if and only if there is a $\pi \in \Pi$ such that $x_{\pi(i)}^{u_i} | x_{\pi(i)}^{v_{\pi(i)}}$ for each $i \in [n]$.  In other words, $\Pi$-divisibility is the Higman partial order of the standard divisibility partial order on the monomials of $\K[X_{[r] \times \P}]$ (viewed as elements of $(\N^r)^*$).  Thus, Higman's Lemma implies that $\Pi$-divisibility is a well-partial-order.  Theorem \ref{main_GB_thm} now implies that $\K[X_{[r] \times \P}]$ has finite Gr\"obner bases; in particular, by Corollary \ref{cor:GBfinitecor} it is a Noetherian $\K[X_{[r] \times \P}] \ast \Pi$-module.
\end{proof}

As a corollary to Theorem \ref{thm:basicpoly}, we also deduce the Noetherian property for ideals that are stable under the action of the symmetric group $\S_\P$.  This was Theorem \ref{thm:sn} from the introduction.

\begin{cor}\label{cor:sn}
The polynomial ring $\K[X_{[r] \times \P}]$ is a Noetherian $\K[X_{[r] \times \P}] \ast \S_\P$-module.
\end{cor}

\begin{proof}
Each polynomial $f \in \K[X_{[r] \times \P}]$ depends on only finitely many column indices.  Thus, if $\pi \in \Pi$, there exists $\sigma \in \S_\P$ such that $\sigma \cdot f = \pi \cdot f$.  Indeed, if the largest column index appearing in $f$ is $m$, then $\sigma$ can be chosen to be the identity on all $i > \pi(m)$.  This implies that every $\S_\P$-stable ideal $I$ is $\Pi$-stable and any $\K[X_{[r] \times \P}] \ast \Pi$ generating set of $I$ is a $\K[X_{[r] \times \P}] \ast \S_\P$ generating set.
\end{proof}

Note that the $r = 1$ version of Corollary \ref{cor:sn} is a main result of \cite{Aschenbrenner2007, Cohen67}.  Our proof and the material in Section 2 constitute a distillation and generalization of the proof in those papers.  A second corollary concerns infinite chains of symmetric ideals, each contained in a finite polynomial ring.

Before stating this result, we must first introduce a filtration of $\K[X_{[r] \times \P}] \ast \Pi$.   Let $Q_n \cong \mathbb{N}^{r \times n}$ be the set of monomials in the polynomial ring $\K[X_{[r] \times [n]}]$, and for $m \geq n$, let $\Pi_{n,m} \subset \Pi$ be the set of functions
\[\Pi_{n,m} = \{ \pi \in \Pi: \ \pi([n]) \subseteq [m] \}.\] 

\begin{thm}\label{thm:pifinite}
The sets $Q_n$ and $\Pi_{n,m}$ form a filtration of $\K[X_{[r] \times \P}] \ast \Pi$.  In particular, every $\Pi$-invariant ascending chain $I_\circ$ stabilizes.
\end{thm}

\begin{proof}
The seven conditions of Definition \ref{def:filtration} are easy to check.  The most difficult to parse is (6), which we describe in detail.   In our setting, condition (6) says that if a monomial $x^u =  x_1^{u_1} \cdots  x_n^{u_n}$ has $u_n \neq 0$ and $\pi \in \Pi$ has $\pi(n) \leq m$, then there is a $\pi' \in \Pi_{n,m}$ such that $\pi'(x^u) = \pi(x^u)$.  But if $\pi \in \Pi$ satisfies $\pi(n) \leq m$, then $\pi \in \Pi_{n,m}$ (since it is increasing).  In particular, we can take $\pi' = \pi$.  The second statement follows from Theorem \ref{thm:symminvarstab} and the fact that $\Pi$-divisibility is a well-partial-order (from the proof of Theorem \ref{thm:basicpoly}).
\end{proof}

The most important and useful implication of Theorem \ref{thm:pifinite} is the following corollary, which concerns chains of ideals stable under the action of the symmetric group.  It is this fact, and its variations, that allow us to prove the theorems in algebraic statistics that appear in the next section.  For simplicity of notation, we write $\S_{n}$ for $\S_{[n]}$ below.

\begin{cor}\label{cor:snfinite}
For each $n \in \P$, let $I_n \in \K[X_{[r] \times [n]}]$ be a $\S_{n}$-invariant ideal.  Suppose that the $I_n$ form an invariant ascending chain:
$$ \S_{m}   I_{n}  \subseteq I_{m}, \ \ \text{for each $n \leq m$}.$$
Then there exists an $n_0 \in \P$ such that for all $n > n_0$,  we have
$$\left< \S_{n}   I_{n_0} \right>_{\K[X_{[r] \times [n]}]}  = I_{n}.$$
In other words, ascending invariant chains are finitely generated up to symmetry.
\end{cor}

\begin{proof}
An ascending invariant chain $I_\circ$ with respect to the filtration of \[\S_{(\P)} := \bigcup_{n \in \P} \S_{n}\] by the $\S_{n}$ is also an ascending invariant chain with respect to the filtration of $\Pi$ by the $\Pi_{n,m}$.  Hence, there exists an $n'_0$ with respect to which each $I_n$ with $n \geq n'_0$ is generated by the generators of $I_{n'_0}$.  Since $\Pi_{n,m} I_n \subseteq \S_{m} I_n$, for all $m \geq n$, this implies that $n_0 = n'_0$ is sufficient in the corollary.
\end{proof} 

Beyond Theorem \ref{thm:basicpoly}, we would like to have more general settings in which there is a priori knowledge that some family of ideals is finitely generated.  In a certain sense, Theorem \ref{thm:basicpoly} is best possible for infinite dimensional polynomial rings, as the following example illustrates.

\begin{ex}\label{ex:notnoth}
The polynomial ring $\K[X_{\P \times \P}]$ is naturally a $\K[X_{\P \times \P}] \ast ( \S_\P \times \S_\P)$-module, but this module is not Noetherian.  For instance, the ideal
$$I  = \left< x_{11} x_{12} x_{22} x_{21}, x_{11} x_{12} x_{22} x_{23} x_{33} x_{31}, \ldots,  x_{11}x_{12} x_{22} \cdots x_{mm} x_{m1},  \ldots \right>$$
is not finitely generated as a $\K[X_{\P \times \P}] \ast ( \S_\P \times \S_\P)$-module.
Via the natural correspondence between square-free monomials in doubly indexed variables and bipartite graphs, the sequence of generators listed above are even length cycles.  The fact that no even length cycle is a subgraph of any other even length cycle implies that this ideal is not finitely generated. \qed
\end{ex}

In spite of Example \ref{ex:notnoth}, it is possible to extend Theorem \ref{thm:basicpoly} via the theory from Section \ref{Pgroebner} to prove that certain ideals in rings such as $\K[X_{\P \times \P}]$ are finitely generated up to the action of $\S_\P \times \S_\P$.  This is done by combining the following elementary proposition with Corollary \ref{cor:divisible} below.  The idea will be to focus on $\Pi$-stable ideals $J  \subseteq \K[X_{\P \times \P}]$ which contain a subideal $I  \subseteq J$ such that $\K[X_{\P \times \P}]/I$ is Notherian (see Example \ref{divisNoetherianexample}).

\begin{prop}\label{prop:quotient}
Suppose that $L \subseteq M \subseteq N$ are $R$-modules, that $L$ is finitely generated, and that $N/L$ is a Noetherian $R$-module.  Then $M$ is finitely generated.
\end{prop}

\begin{proof}
Since $N/L$ is Noetherian, $M/L$ has a finite generating set, with representatives in $M$.  These generators along with the generators of $L$ form a finite generating set of $M$.
\end{proof}

We next consider a natural class of rings that inherit Noetherianity from being contained in a Noetherian semigroup ring.  The goal in applications will be to show that quotients as above are isomorphic to one of these special rings.

\begin{defn}
A subsemigroup ring $\K[Q'] \subseteq \K[Q]$ is called \emph{divisible} if $q_1, q_2 \in Q'$ and $q_1 = q_3q_2$ implies that that $q_3 \in Q'$.  The subsemigroup ring $\K[Q']$ is $P$-\emph{invariant} if for all $q \in Q'$ and $p \in P$, we have $pq \in \K[Q']$.  
\end{defn}

\begin{cor}\label{cor:divisible}
Let $\K[Q']$ be a divisible $P$-invariant subring of $\K[Q]$.   If $\preceq$ is a $P$-order on $Q$, then
\begin{enumerate}
\item $\preceq$ is a $P$-ordering on $Q'$.
\end{enumerate}
If, in addition, $P$-divisibility is a well-partial-ordering on $Q$, then
\begin{enumerate}
\item[(2)]  $P$-divisibility is a well-partial-ordering on $Q'$ and
\item[(3)] $P$-invariant ideals of $\K[Q']$ have  finite $P$-Gr\"obner bases.
\end{enumerate}
If, in addition $Q_n$ and $P_{n,m}$ are a filtration of $\K[Q]\ast P$, then
\begin{enumerate}
\item[(4)]  $Q'_n = Q' \cap Q_n$ and $P_{n,m}$ are a filtration $\K[Q'] \ast P$, and
\item[(5)]  invariant chains $I_\circ$ with $I_n \in \K[Q_n']$ stabilize.
\end{enumerate}
\end{cor}

\begin{proof}
(1) If $q_1, q_2 \in Q'$, then $Q' \subseteq Q$ implies that the condition of Proposition \ref{prop:orderchar} is satisfied.  In particular, (1) holds regardless of whether $Q'$ is divisible.

(2)  Consider any infinite sequence of monomials in $Q'$.  Since $P$-divisibility is a well-partial-ordering on $Q$, this sequence is good when considered as a subset of $Q$.  Since $Q' \subseteq Q$ is a divisible subsemigroup, the sequence is also good in $Q'$.  Proposition \ref{prop:equivorder} implies that $P$-divisibility is also a well-partial-order on $Q'$.

(3) This follows from (2) and Theorem \ref{main_GB_thm}.

(4) There are seven conditions to check in the definition of a filtration; all of them are straightforward.

(5)  This follows from (4) and Theorem \ref{thm:symminvarstab}. 
\end{proof}

We close this section with an example illustrating how Proposition \ref{prop:quotient} and Corollary \ref{cor:divisible} will be used in Section \ref{sec:algstat}.

\begin{ex}\label{divisNoetherianexample}
Consider $\K[X_{\P \times \P}]$ as a module over $\K[X_{\P \times \P}] \ast \Pi$ with $\Pi$ acting on both indices simultaneously (i.e. $\pi x_{i,j} = x_{\pi(i), \pi(j)}$).  By Example \ref{ex:notnoth}, $\K[X_{\P \times \P}]$ is not a Noetherian $\K[X_{\P \times \P}] \ast \Pi$-module.  However, consider a $\Pi$-stable ideal $J \subseteq \K[X_{\P \times \P}]$ such that $I \subseteq J$, where 
$$I  =  \left< 
\det \begin{pmatrix} 
x_{i_1, j_1} & x_{i_1, j_2} \\
x_{i_2, j_1} & x_{i_2, j_2}
\end{pmatrix} :  i_1, i_2, j_1, j_2 \in \P \right>_{\K[X_{\P \times \P}]}$$
is the ideal of two-by-two minors of the matrix $X_{\P \times \P}$.  We have the following isomorphism of $\K[X_{\P \times \P}] \ast \Pi$-modules:
$$\K[X_{\P \times \P}]/ I \cong \K[y_{1,i}y_{2,j} : i, j \in \P],$$ the map being induced by $x_{i,j} \mapsto y_{1,i}y_{2,j}$.  Thus, $\K[X_{\P \times \P}]/ I$ has the same module structure as that of $\K[y_{1,i}y_{2,j} : i, j \in \P]$ as a $\K[y_{1,i}y_{2,j} : i, j \in \P] \ast \Pi$-module.  
Since $\K[y_{1,i}y_{2,j} : i, j \in \P]$ is a $\Pi$-stable divisible semigroup ring that is a subring of $\K[Y_{[2] \times \P}]$, we see that 
$\K[X_{\P \times \P}]/ I$ is a Noetherian $\K[X_{\P \times \P}] \ast \Pi$-module by  Corollary \ref{cor:divisible}.  Since $I$ is also finitely generated as a $\K[X_{\P \times \P}] \ast \Pi$-module, it follows that $J$ is finitely generated by Proposition \ref{prop:quotient}.
\qed
\end{ex}


\section{Applications in Algebraic Statistics}\label{sec:algstat}

In this section, we apply the theory developed in the previous two sections to give new proofs of some classical finiteness results about Markov bases of hierarchical models \cite{Hosten2007, Santos2003}.  These finiteness theorems follow from  Corollary \ref{cor:snfinite} for finite generation of chains of increasing symmetric ideals.  We also extend these results using Proposition \ref{prop:quotient} and Corollary \ref{cor:divisible} to give an affirmative solution to the independent set conjecture \cite[Conj.~4.6]{Hosten2007}.  Finally, we explain how these finiteness results extend beyond hierarchical models to more general statistical models.

Let $r_1,  \ldots, r_m \in \P$ and for a subset $F \subseteq [m]$, set $R_F = \prod_{i \in F} [r_i]$ to be the Cartesian product of the index sets $[r_i]$.  If $F = [m]$, we use the shorthand $R = R_F$.  For an algebraic object $\mathbb{A}$ (e.g.~a field, semiring, monoid) and a finite set $M$, let $\mathbb{A}^M$ be the Cartesian product of $\mathbb{A}$ with itself $\#M$ times, with coordinates indexed by $M$.

Let ${\bf i} \in R$ denote an index vector.  For each $F \subset [m]$, let ${\bf i}_F := (i_f)_{f \in F}$ be the substring ${\bf i}_F \in R_F$ obtained from ${\bf i} $ by the natural projection.
For $u \in \mathbb{R}^R$ and $F \subseteq [m]$, also let $u|_F \in \mathbb{R}^{R_F}$ be the $F$-\textit{marginal} of $u$, defined by linearly extending 
$$e_{{\bf i}}|_{F}  :=  e_{{\bf i}_F}.$$
Here, $e_{{\bf i}}$ is the standard unit table in $\mathbb{R}^{R}$, having a $1$ in the ${\bf i}$ position and zero elsewhere,  and similarly  $e_{{\bf i}_F}$ is the standard unit table in $\mathbb{R}^{R_F}$.  

Given a collection $\Gamma = \{F_1, F_2, \ldots \}$ of subsets of $[m]$, we define the $\Gamma$-\emph{marginal map} by
$$\pi_{\Gamma,r} :  \mathbb{R}^{R}  \to  \bigoplus_{F \in \Gamma}  \mathbb{R}^{R_F}$$
$$u  \mapsto  (u|_{F_1}, u|_{F_2}, \ldots ).$$
From the linear transformation $\pi_{\Gamma,r}$, we can extract the matrix $A_{\Gamma,r}$ representing it.  This matrix $A_{\Gamma,r}$ is called the \textit{design matrix} of the hierarchical model associated to $\Gamma$ in algebraic statistics \cite{Drton2009}. 

Associated to any linear transformation $A:  \Z^r \to \Z^d$ is the lattice $\ker_\Z A$.  Among the many important spanning sets for a lattice are the Markov bases, which are special sets that allow one to take random walks over the fibers $(u + \ker_\Z A) \cap \N^r$.

\begin{defn}
A \emph{Markov basis} for the matrix $A$ (or lattice $\ker_\Z A$) is a finite subset $\mathcal{B} \subset \ker_\Z(A)$ such that for all $u, v \in \N^r$ with $A u = A v$, there exists a sequence $b_1, \ldots, b_L \in \mathcal{B}$ such that
$$u = v + \sum_{i =1}^L b_i  \quad \quad \mbox{and} \quad \quad  v +  \sum_{i =1}^l b_i  \in \N^r, \quad l = 1, 2, \ldots, L.$$
Elements of a Markov basis are called \emph{moves}.  A Markov basis for $A$ is \emph{minimal} if no proper subset is a Markov basis of $A$.  
\end{defn}

Markov bases of the matrices $A_{\Gamma,r}$ are useful for performing statistical hypothesis tests by running random walks over contingency tables (see \cite{Diaconis1998} or Chapter 1 in \cite{Drton2009}).  Note that Markov bases are not in general unique, even if we assume the Markov basis is minimal.

One of the main mathematical questions  about Markov bases of hierarchical models is the following: How does the structure of the Markov basis depend on $\Gamma$ and $r_1, \ldots, r_m$?   A specific problem of this type is to determine what finiteness properties of the Markov bases should be expected  when we fix $\Gamma$ and send one or more values of $r_i \to \infty$. Questions about finiteness for these Markov bases are natural in our setting because the lattice $\ker_\mathbb{Z} A_{\Gamma,r}$ is stable under the action of the product of symmetric groups $\S_{r_1} \times \cdots \times \S_{r_m}$, where $\S_{r_i}$ acts by permuting the $i$th index.  Furthermore, given any $\Gamma$, $r \in \P^{m}$, and $t \in \N^{m}$, a vector $b \in \ker_{\Z} A_{\Gamma,r}$ can be naturally lifted into $\ker_{\Z} A_{\Gamma,r+t}$ by padding with zeroes.  Denote the resulting vector by ${\rm pad}_{r+t}(b)$.

We now make precise the notion of sending some $r_{i} \to \infty$.  First, fix a collection of indices $T \subseteq [m]$ which will ``go to infinity''.  For each fixed set of values $r_{i}$ with $i \in [m] \setminus T$, we consider the Markov bases of the matrices $A_{\Gamma,r}$ (here, $r_{i}$ is allowed to be arbitrary when $i \in T$).  We have \emph{finite Markov bases up to symmetry} in this situation if for every fixed set of values $r_{i}$ with $i \in [m] \setminus T$, there exist $r_{i}$ with $i \in T$ and a finite set of moves $\mathcal{B} \subseteq \ker_{\Z} A_{\Gamma,r}$, such that for all $t \in \N^{m}$ with $t_{i}= 0$ for $i \in [m]\setminus T$, the set  
$$\S_{r_1+t_{1}} \times \cdots \times \S_{r_m+t_{m}} \cdot \{ {\rm pad}_{r+t}(b) : b \in \mathcal{B} \} $$
is a Markov basis for $A_{\Gamma,r+t}$.  Otherwise, there is no finite Markov basis up to symmetry.

We represent two examples illustrating that in some situations the Markov basis is finite up to symmetry and in other cases it is not.

\begin{ex}
Let $\Gamma = \{ \{1\}, \{2\} \}$.  Then $\pi_\Gamma:  \Z^{[r_1] \times [r_2]} \to \Z^{[r_1]} \oplus\Z^{[r_2]}$ is the map that computes the row and column sums of an $r_1 \times r_2$ table.  Thus $\ker_\Z A_\Gamma$ consists of all integral tables whose row and column sums are equal to zero.

If both $r_1, r_2 \geq 2$, the Markov basis for this model consists of  the $2 {r_1 \choose 2} {r_2 \choose 2}$ moves:
$$\mathcal{B} =  \left\{ e_{i_1 j_1} + e_{i_2 j_2} - e_{i_1 j_2} - e_{i_2 j_1} : i_1, i_2 \in [r_1], j_1, j_2 \in [r_2] \right\}.$$
For example, for $r_1 = 3, r_2 = 4$, a typical element in $\mathcal{B}$ is the $3 \times 4$ table
$$
\begin{pmatrix}
1 & 0 & -1 & 0 \\
-1 & 0 & 1 & 0 \\
0 & 0 & 0 & 0 \\
\end{pmatrix}.
$$
Up to the natural action of $\S_{r_1} \times \S_{r_2}$, permuting rows and columns of the matrices,  there is only one move in the Markov basis \cite{Diaconis1998}. \qed
\end{ex}

On the other hand, these types of finite Markov basis descriptions are known not to hold for general $\Gamma$ when we let many of the numbers $r_i \to \infty$.

\begin{ex}\label{ex:no3way}
Let $\Gamma = \{\{1,2\},\{1,3\},\{2,3\} \}$ be the three cycle hierarchical model (also called the model of no 3-way interaction).  Then $\pi_\Gamma : \Z^{[r_1] \times [r_2] \times [r_3]} \to \Z^{[r_1] \times [r_2]} \oplus \Z^{[r_1] \times [r_3]} \oplus \Z^{[r_2] \times [r_3]}$ is the map that computes all $2$-way marginals of the three way table $u$.  For all $m$, the move
$$ \sum_{i =1}^m (e_{i,i,1} - e_{i,i,2})  + e_{m,1,2} - e_{m,1,1} + \sum_{i = 1}^{m-1} (e_{i,i+1,2} - 
e_{i,i+1,1}) $$
belongs to every minimal Markov basis for $\Gamma$ for all $r_1, r_2 \geq m$ and $r_3 \geq 2$  \cite{Diaconis1998}.  When $r_{3}= 2$, these Markov basis elements can be represented as two $r_{1} \times r_{2}$ matrices obtained from extracting slices where $i_{3} =1$ and $i_{3} = 2$ respectively.  When $r_{1}= r_{2} = 5$, the Markov basis element is: 
\begin{equation*}
\begin{pmatrix}
1 & -1 & 0  &  0 &0  \\
0 & 1 & -1 &  0 & 0 \\
0 & 0 & 1 & -1 & 0  \\
0 & 0 & 0 &  1 & -1 \\
-1 & 0 & 0 & 0 & 1 
\end{pmatrix}
\quad
\begin{pmatrix}
-1 & 1 & 0 &   0 &0  \\
0 & -1 & 1 &  0 & 0 \\
0 & 0 & -1 &  1 & 0  \\
0 & 0 & 0 &  -1 & 1 \\
1 & 0 & 0 &  0 & -1 
\end{pmatrix}.
\end{equation*}

In particular, Markov bases for $\Gamma$ are not finite up to $\S_{r_{1}} \times \S_{r_{2}}\times \S_{r_{3}}$ symmetry on $r_{1} \times r_{2} \times r_{3}$ arrays for $r_3 \geq 2$ as $r_1$ and  $r_2$ both tend to infinity.  \qed
\end{ex}

These two examples illustrate a dichotomy between cases where we send more than one of the $r_i \to \infty$.  In some situations the Markov basis is finite up to symmetry, and in other cases it is not.   If we only send one of the $r_i$ to infinity, however, there is always a finite Markov basis up to symmetry \cite{Hosten2007,Santos2003}:

\begin{thm}\label{thm:finitemarkov}
For any $\Gamma$ and fixed $r_1, \ldots, r_{m-1}$, there exists an $N = N(\Gamma; r_1, \ldots, r_{m-1})$ such that the Markov basis for $A_\Gamma$ for $r_1,  \ldots r_m$ with $r_m > N$ is determined up to symmetry by the Markov basis for $r_1, r_2, \ldots, r_{m-1}, N$.  
\end{thm}

We provide a new proof of Theorem \ref{thm:finitemarkov} below.  An important ingredient will be the fundamental theorem of Markov bases, which translates questions about Markov bases into questions about generating sets of toric ideals.

Given any matrix $A = (a_{ij}) \subseteq \mathbb{Z}^{k \times r}$, consider the ring homomorphism:
$$\phi:  \mathbb{K}[x_1, \ldots, x_r]  \rightarrow \mathbb{K}[y_1^{\pm 1}, \ldots, y_k^{\pm 1}], \quad \quad  x_j \mapsto  \prod_{i =1}^k  y_i^{a_{ij}}.$$
The kernel of $\phi$ is the \textit{toric ideal} $I_A := \ker \phi.$
The ideal $I_A$ is a prime ideal that gives an algebraic encoding of the integer kernel of the matrix $A$  since
$$I_A  = \left<  x^u - x^v : u,v \in \N^r, \, \,  Au =Av \right>.$$
Note that $\K[X]/I_{A}$ is a semigroup ring, the ring generated by the monomials $\phi(x_{1}), \ldots, \phi(x_{r})$.

The following theorem establishes the connection between Markov bases of the lattice $\ker_\Z A$ and the toric ideal $I_A$.  (Below, the vectors $b^{+} \in \N^r$ and $b^{-} \in \N^r$ are the nonnegative and nonpositive part, respectively, of $b = b^{+} - b^{-} \in \mathcal{B}$).

\begin{thm}[Fundamental theorem of Markov bases] \cite{Diaconis1998}
A finite subset $\mathcal{B} \subseteq \ker_\mathbb{Z} A$ is a Markov basis for $A$ if and only if the set of binomials
$$\left\{  x^{b^+} - x^{b^-} :  b \in \mathcal{B} \right\}$$
is a generating set of the toric ideal $I_A$.
\end{thm}

\begin{proof}[Proof of Theorem \ref{thm:finitemarkov}]
Applying the fundamental theorem of Markov bases, it suffices to show that the associated toric ideals are finitely generated up to symmetry.
For each value of $r_m \in \P$, let $A_{r_m}$ be the matrix representing the linear transformation $\pi_\Gamma$ for a table of size $r_1, \ldots, r_m$.  That is $A_{r_{m}} = A_{\Gamma,r}$, but where we are paying special attention to the changing value of $r_{m}$.  Each of the ideals $I_{A_{r_m}}$ is contained in $\K[X_{R_{[m-1]} \times [r_m]}]$. Taking $k = \prod_{i =1}^{m-1} r_i$ and identifying $\prod_{i =1}^{m-1} [r_i]$ with $[k]$, we see that each ideal naturally lies in $\K[X_{[k] \times [r_m]}]$.  Furthermore, each ideal $I_{A_{r_m}}$ is stable under the action of $\S_{r_m}$.  On tables, $\S_{r_m}$ acts by permuting ``slices'' of the table.  The ideals are also nested:  
$$  \left< \S_{r_{m}+1}  I_{A_{r_{m}}} \right>_{\K[X_{[k] \times[r_{m}+1]}]} \subset I_{A_{r_{m}+1}},$$
which on the level of tables corresponds to the fact that we can always add a slice of all zeroes to an element $b \in \ker_\Z A_{r}$ and obtain an element $b' \in \ker_\Z A_{r+1}$.  Thus, the sequence of ideals $I_{A_1}, I_{A_2}, \ldots$ forms an ascending invariant chain.  Therefore, by Corollary \ref{cor:snfinite} they have a finite generating set up to the filtration of $\S_{(\mathbb{P})}$ by  the $\S_{r_m}$. 
\end{proof}

Our new proof of Theorem \ref{thm:finitemarkov} has the advantage over the proofs from \cite{Hosten2007, Santos2003} that it puts these finiteness properties into a very general framework.  On the other hand, the older proofs produce bounds on the number $N(\Gamma; r_1, \ldots, r_{m-1})$.  Part of the reason for introducing our more general framework is that it can produce finiteness results in situations where the ideas from \cite{Hosten2007, Santos2003} do not generalize.   The technique in \cite{Hosten2007, Santos2003} is to show that the universal Gr\"obner basis is finite up to $\S_\P$ symmetry, which implies finite generation up to symmetry (a \textit{universal Gr\"obner basis} is a set of polynomials that is a Gr\"obner basis with respect to every term order).  That idea does not work in the more general settings considered below because the universal Gr\"obner basis is not, in general, finite up to symmetry (e.g.~the universal Gr\"obner basis of the ideal of $2\times 2$ minors in $\K[X_{[k] \times [k]}]$ requires polynomials of degree $k$).

Example \ref{ex:no3way} shows that there cannot be a general finiteness result when two or more of the $r_i$ are sent to infinity.  However, we can still produce finiteness theorems when some of the $r_i \to \infty$ and $\Gamma$ satisfies some extra properties.

\begin{defn}
A subset $T \subseteq [m]$ is called an \emph{independent subset} of $\Gamma$ if $\#(T \cap F) \leq 1$ for all $F \in \Gamma$.
\end{defn}

Equivalently, the independent subsets of $\Gamma$ are precisely the independent sets of the $1$-skeleton of $\Gamma$ (that is, of the underlying graph).

The main theorem of this section is a finiteness property for Markov bases in models $\Gamma$ that have independent vertex sets.  This provides a proof of the independent set conjecture of Ho\c{s}ten and the second author \cite[Conj. 4.6]{Hosten2007}. 

\begin{thm}\label{thm:independentset}
Let $\Gamma \subseteq 2^{[m]}$, and suppose that $T \subseteq [m]$ is an independent set of $\Gamma$.  Fix the table dimensions $r_s$ such that $s \in [m] \setminus T$.  Then $A_\Gamma$ has a finite Markov basis up to the natural action of $\S_{r_1} \times \cdots \times \S_{r_m}$ as $r_t \to \infty$ for all $t \in T$. 
\end{thm}

Proving Theorem \ref{thm:independentset} requires two intermediate results.  First of all, we shall need to understand the relationships between toric ideals $I_{A_\Gamma}$ for varying $\Gamma$.  Secondly, we will need to understand an important family of $\Gamma$ that are called decomposable.

One simplification we can make about $\Gamma$ is to assume it is a simplicial complex; that is, if $S \in \Gamma$ and $T \subseteq S$ then $T \in \Gamma$ as well.  We may make this assumption without loss of generality since ``the marginal of a marginal is a marginal''. In other words, adding $T$ to $\Gamma$ when $S \in \Gamma$ and $T \subseteq S$ does not change $\ker A_\Gamma$.

\begin{lemma}\label{lem:contain}
Suppose that $\Gamma_1 \subseteq \Gamma_2$, in the sense that for each $S \in \Gamma_1$, there is a $T \in \Gamma_2$ such that $S \subseteq T$.  Then $\ker A_{\Gamma_2}  \subseteq \ker A_{\Gamma_1}$ and the toric ideals satisfy $I_{A_{\Gamma_2}} \subseteq I_{A_{\Gamma_1}}$.
\end{lemma}

\begin{proof}
If $S \subseteq T$, then $u|_S  =  (u|_T) |_S$.  Thus, if $\Gamma_1 \subseteq \Gamma_2$, the marginal map $\pi_{\Gamma_1}$ factors through $\pi_{\Gamma_2}$.
\end{proof}

A simplicial complex $\Delta$ has a \emph{reducible decomposition} $(\Delta_1, S, \Delta_2)$ if $\Delta = \Delta_1 \cup \Delta_2$, $\Delta_1 \cap \Delta_2 = 2^{S}$ (where $2^S$ is the power set of $S$), and neither $\Delta_1$ nor $\Delta_2 = 2^S$.  A simplicial complex with a reducible decomposition is called \emph{reducible}.  A simplicial complex is \emph{decomposable} if it is either a simplex (of the form $2^K$) or it is reducible and both $\Delta_1$ and $\Delta_2$ are decomposable.  The following theorem characterizes the generating sets of the toric ideals $I_{A_\Gamma}$ whenever $\Gamma$ is a decomposable simplicial complex.

\begin{thm}\cite{Dobra2003, Takken2000}\label{thm:decomp}
If $\Gamma$ is a decomposable simplicial complex, then $I_{A_\Gamma}$ is generated by quadratic binomials.  As $r_1, \ldots, r_m \to \infty$, there is a finite set of quadratic binomials that generate $I_{A_\Gamma}$ up to the action of $\S_{r_1} \times \cdots \times \S_{r_m}$.  Furthermore, let $T \subseteq [m]$ and fix the table dimensions $r_s$ where $s \in [m] \setminus T$, and let $r_t = r$ for all $t \in T$.  Let $\S_{r}$ act diagonally on $[r]^{\#T}$.  Then, the generators of $I_{A_\Gamma}$ stabilize up to the action of $\S_{r}$ after $r \geq 2 \#T$.
\end{thm}

\begin{proof}[Proof of Theorem \ref{thm:independentset}]
By the fundamental theorem of Markov bases, it suffices to show that the corresponding toric ideals $I_{A_\Gamma}$ are finitely generated up to symmetry.
It also suffices to show the finiteness result when considering the action of a much smaller group contained inside of $\S_{r_1} \times \cdots \times \S_{r_m}$.  Namely, we will send $r_t \to \infty$ for $t \in T$ simultaneously and consider the diagonal action of $\S_r$ acting on the indices $i_{t}$ with $t \in T$.  This is sufficient because every Markov basis move for a small table embeds as a Markov basis element for a table of larger dimensions, by the padding operation.

For each $r \in \P$, let $I_{A_r}$ be the corresponding toric ideal, which belongs to the ring 
$$\K[Q_r] := \K[X_{R_{[m] \setminus T} \times [r]^{\#T}}].$$
Let $\K[Q]$ denote the limiting ring
$$\K[Q] := \K[X_{R_{[m] \setminus T} \times \P^{\#T}}].$$
Let $\Pi$ act on $\K[Q]$ by acting diagonally on $\P^{\#T}$.  Then the $Q_r$ and $\Pi_{n,r}$ form a filtration of $\K[Q] \ast \Pi$, and the sequence of ideals $I_\circ = I_{A_1} \subseteq I_{A_2} \subseteq \cdots $ is an invariant chain.  Let $J_\Gamma = \mathcal{N}(I_\circ) = \cup_{n\geq1} I_{A_n}$.  Our goal is to show that the chain $I_\circ$ stabilizes.

Consider the following decomposable simplicial complex:
$$\Gamma' =  \left\{  ([m] \setminus T) \cup \{t \} : t \in T \right\} \cup 2^{[m] \setminus T}.$$
For each $r \in \P$, let $I_{B_r}$ be the toric ideal $I_{\Gamma'}$ which is in the ring $\K[Q_r]$.  The $I_{B_r}$ form a chain with respect to the filtration of $\K[Q] \ast \Pi$.  Since $\Gamma'$ is decomposable, this chain stabilizes by Theorem \ref{thm:decomp}. Let $J_{\Gamma'} \subseteq \K[Q]$ denote the union of this chain.
Since $T$ is an independent set of $\Gamma$, we have $\Gamma \subseteq \Gamma'$, which implies  
$I_{B_r} \subseteq I_{A_r}$ by Lemma \ref{lem:contain}.  We now want to apply Corollary \ref{cor:divisible} and Proposition \ref{prop:quotient}  to deduce that the chain $I_\circ$ stabilizes.

For each $r \in P$, the ideal $I_{B_{r}}$ is a toric ideal, and hence $\K[Q_{r}]/I_{B_{r}}$ is a semigroup ring.  The limiting ring $\K[Q]/J_{\Gamma'}$ is also a semigroup ring, and it is generated by all monomials appearing in the ring homomorphisms $\phi$ associated to the matrices $B_{r}$.  This can be explicitly obtained by looking at the effect of the linear transformation $\pi_{\Gamma'}$ on standard unit vectors.  Let $S = [m] \setminus T$.  Then, 
$$
\pi_{\Gamma'}(e_{\bf i}) = \oplus_{t \in T} e_{{\bf i}_{S \cup \{t\}}} \in \bigoplus_{t \in T} \R^{R_{S} \times [r]}.
$$
For each $F \in \Gamma'$ and ${\bf j} \in R_{S} \times [r]$ we have a variable $y^{F}_{\bf j}$.  The formula for $\pi_{\Gamma'}$ implies that for each ${\bf i } \in R$, 
$$\phi(x_{\bf i})  =  \prod_{t \in T} y^{S \cup \{t\}}_{{\bf i}, i_{t}}.$$
This implies that
$$\K[Q]/J_{\Gamma'} =: \K[Q']  \cong \K\left[ \prod_{t \in T} y_{{\bf i}, t, j_t}:  {\bf i} \in R_{S}, \, \,   j_t \in \P \right].$$
In particular, $\K[Q']$ is a subsemigroup ring of $\K[Y_{ R_{[m] \setminus T} \times T \times \P }]$.  (This is obtained by replacing the cumbersome superscript $S \cup \{t\}$ with a simple $t$ subscript.)  

We now show that $\K[Q']$ is a divisible subsemigroup ring of $\K[Y_{ R_{[m] \setminus T} \times T \times \P }]$.  Consider the $\K$-algebra homomorphism $\psi$  from $\K[Y_{ R_{[m] \setminus T} \times T \times \P }]$ to $\K[Z_{ R_{[m] \setminus T} }]$ that maps $y_{{\bf i}, t, j_t}$ to $z_{\bf i}$.  A monomial $y^{\alpha} \in \K[Y_{ R_{[m] \setminus T} \times T \times \P }]$ belongs to $\K[Q']$ if and only if $\psi(y^{\alpha})$ is of the form $(z^{\beta})^{\#T}$ for some $\beta$. Now, if $\psi(y^{\alpha^{1}}) = (z^{\beta^{1}})^{\#T}$, $\psi(y^{\alpha^{2}}) = (z^{\beta^{2}})^{\#T}$, and $y^{\alpha^{1}}|y^{\alpha^{2}}$  then  $\psi(y^{\alpha^{2} - \alpha^{1}}) = (z^{\beta^{2} -\beta^{1}})^{\#T}$.  This implies that $\K[Q']$ is a divisible subsemigroup of $\K[Y_{ R_{[m] \setminus T} \times T \times \P }]$.

Letting $\Pi$ act on $\P$ and since $R_{[m] \setminus T} \times T$ is a finite set, we have that $\Pi$-divisibility on $\K[Y_{ R_{[m] \setminus T} \times T \times \P }]$ is a well-partial-ordering.  Then, Corollary \ref{cor:divisible} implies that $\Pi$-divisibility is also a well-partial-ordering on $\K[Q']$.

Consider the filtration  on $\K[Y_{ R_{[m] \setminus T} \times T \times \P }] \ast \Pi$ using $\K[\tilde{Q}_r] = \K[Y_{ R_{[m] \setminus T} \times T \times [r] }]$ with $\Pi_{n,m}$.  Let $Q'_r  = \tilde{Q}_r \cap Q'$ be the induced filtration in $\K[Q'] \ast \Pi$.  Corollary \ref{cor:divisible} also implies that chains with respect to this filtration stabilize.  In particular, the chain $I_{A_r}/I_{B_r}$ stabilizes.  That is, there is an $r_0 \geq 2$ and a finite generating set $\mathcal{F}$ of $I_{A_{r_0}}/ I_{B_{r_0}}$ such that $\Pi_{r,r_0} \mathcal{F}$ generates  $I_{A_r}/I_{B_r}$ for all $r > r_0$.  Since $r_0 \geq 2$,  Theorem \ref{thm:decomp} implies $I_{B_{r_0}}$  is generated by quadrics which also generate all $I_{B_r}$ up to the action of $\Pi$.  Finally, Proposition \ref{prop:quotient} implies that $\left< \Pi_{r_0, r}I_{A_{r_0}} \right>_{\K[Q_r]} = I_{A_r}$, which proves the theorem.
\end{proof}

\begin{ex}[6-cycle]
The six cycle $\Gamma = \{ \{1,2\},\{2,3\},\{3,4\},\{4,5\},\{5,6\},\{1,6\} \}$ has the independent set $T = \{2,4,6\}$. If we fix $r_1, r_3, r_5$ and send $r_2, r_4, r_6 \to \infty$, then there will be a finite Markov basis for $A_\Gamma$ up to the natural action of the symmetric group.  \qed
\end{ex}

Theorems \ref{thm:finitemarkov} and \ref{thm:independentset} are finiteness results for Markov bases, but it is also natural to extend these ideas to other statistical situations.  Indeed, the Markov bases under consideration are useful tools for studying hierarchical models.  As sets, these models are families of probability distributions inside the probability simplex 
$$\Delta_{R}  = \left\{ p \in \mathbb{R}^{R} :  \sum_{i \in R} p_i = 1, p_i \geq 0, i \in R \right\}.$$
Each point $p \in \Delta_{R}$ is a probability distribution for an $m$-dimensional discrete random vector $Y= (Y_1, \ldots, Y_m)$ with state space equal to $R$. The $i$th coordinate is the probability of the event $Y = i$, and $p_i = P(Y = i)$.

The hierarchical model $\mathcal{M}_\Gamma$ is defined as the set $\mathcal{M}_\Gamma = V(I_{A_\Gamma}) \cap \Delta_{R}$ of solutions to the toric ideal $I_{A_\Gamma}$ inside the probability simplex.  Turning this around, the homogeneous vanishing ideal $\mathcal{I}^h(\mathcal{M}_\Gamma) = I_{A_\Gamma}$ encodes an implicit description of the model that is finite up to symmetry as the number of states of some of the random variables go to infinity.

Using reasoning similar to that found in the preceding proofs, one can deduce finiteness for the implicit representations of families of statistical models as the number of states of some of the variables tend to infinity.  We give brief proofs, which follow the same outlines as those of Theorems \ref{thm:finitemarkov} and \ref{thm:independentset}.

\begin{thm}
For each $r \in \P$, let $\mathcal{M}_r \subseteq \Delta_{R\times [r]}$, where $R = \prod_{i =1}^m  [r_i]$, be a statistical model for $m+1$ dimensional discrete random vectors.  Suppose that each homogeneous vanishing ideal $I_r = \mathcal{I}^h( \mathcal{M}_r) \subseteq \mathbb{R}[X_{R \times [r] }]$ is stable under the action of $\S_r$, and that for each $r$, we have $I_{r} \subseteq I_{r+1}$.  Then, up to symmetry there is a finite set of polynomials that generates the ideals $I_r$ for all $r$.
\end{thm}

\begin{proof}
The sequence of ideals $I_{1}, I_{2}, \ldots$ forms an ascending invariant chain.  Therefore, by Corollary \ref{cor:snfinite} they have a finite generating set up to the filtration of $\S_{(\mathbb{P})}$ by  the $\S_{r_m}$.
\end{proof}

For each $r \in \P^n$, let $\mathcal{M}_r \subseteq \Delta_{\mathcal{S} \times R}$, where $\mathcal{S} = [s_1] \times \cdots \times [s_m]$ and $R = [r_1] \times \cdots \times [r_n]$, be a statistical model for an $m+n$ dimensional discrete random vector $(Y,Z) = (Y_1, \ldots, Y_m, Z_1, \ldots, Z_n)$.  Suppose that each homogeneous vanishing ideal $I_r = \mathcal{I}^h( \mathcal{M}_r) \subseteq \mathbb{R}[X_{\mathcal{S} \times R }]$ is stable under the action of $\S_{r_1} \times \cdots \times \S_{r_n}$ and assume that for each $r \in \P^n$ and $ t \in \N^n$, we have $I_r \subseteq I_{r + t}$.    To generalize Theorem \ref{thm:independentset} to arbitrary statistical models, we need to explain what should be meant by the condition that a collection of vertices forms an independent set.  The simplest (algebraic) way to guarantee such a generalization is to require that for each $r$, we have $I_{B_{r}} \subseteq I_{r}$, where $I_{B_{r}}$ is the toric ideal of the hierarchical model whose simplicial complex  $\Gamma$ has facets $\{ [m] \cup \{i' \} : i' \in \{1', 2', \ldots, n'\}$.  Note that this is the same ideal appearing in the proof of Theorem \ref{thm:independentset}.


In more statistical language, the condition $I_{B_{r}} \subseteq I_{r}$ for all $r$ is equivalent to the random vector $(Y,Z)$ satisfying the conditional independence statement  
$Z_1 \ind Z_2 \ind \cdots \ind Z_n | Y$ (see Chapter 3 in \cite{Drton2009} for connections between conditional independence and hierarchical/graphical models).  We state our result using the language of conditional independence.  

\begin{thm}
For each $r \in \P^n$, let $\mathcal{M}_r \subseteq \Delta_{\mathcal{S} \times R}$, where $\mathcal{S} = [s_1] \times \cdots \times [s_m]$ and $R = [r_1] \times \cdots \times [r_n]$, be a statistical model for an $m+n$ dimensional discrete random vector $(Y,Z) = (Y_1, \ldots, Y_m, Z_1, \ldots, Z_n)$.  Suppose that each homogeneous vanishing ideal $I_r = \mathcal{I}^h( \mathcal{M}_r) \subseteq \mathbb{R}[X_{\mathcal{S} \times R }]$ is stable under the action of $\S_{r_1} \times \cdots \times \S_{r_n}$ and assume that for each $r \in \P^n$ and $ t \in \N^n$, we have $I_r \subseteq I_{r + t}$.  If, in addition, the $\mathcal{M}_r$ all satisfy the conditional independence constraint $Z_1 \ind Z_2 \ind \cdots \ind Z_n | Y$, then up to symmetry there is a finite set of polynomials that generates the ideals $I_r$ for all $r$.
\end{thm}   

\begin{proof}
The key feature of this theorem is the conditional independence constraint $$Z_1 \ind Z_2 \ind \cdots \ind Z_n | Y.$$  Let $\Gamma'$ be the simplicial complex with facets $  [m] \cup \{i'\}$ such that $i' \in \{1', 2', \ldots, n'\}$; this is the decomposable simplicial complex that appeared in the proof of Theorem \ref{thm:independentset}.  The conditional independence statement holding for the model $\mathcal{M}_{r}$ is equivalent to $I_{B_{r}} \subseteq I_{r}$ (see Chapter 3 in \cite{Drton2009}).  The remainder of the proof now follows closely that of Theorem \ref{thm:independentset}.
\end{proof}


\section{Further Directions}\label{sec:further}

From the standpoint of computational algebra, we have proved  theorems asserting the existence of finite generating sets of ideals up to symmetry.  Many open problems remain about how to transition from these existence theorems to effective versions and, in particular, how to develop specific algorithms for computing with symmetric ideals.  We  outline some of these challenges here.

Many chains of ideals in algebraic statistics arise as kernels of ring homomorphisms.  Besides knowing that these chains eventually stabilize and have finite generating sets, one desires upper bounds on when stabilization occurs in terms of the input data.  To be more precise, for each $r \in \P$, let $\phi_r : \K[X_{[k] \times [r]}] \to R$ be a ring homomorphism and let $I_r = \ker \phi_r$.  Suppose that each $I_r$ is invariant under the action of $\S_{r}$ and that  this sequence of kernels is nested: $I_r \subseteq I_{r+1}$.  We call such a chain a \emph{chain of kernels}.  

\begin{ques}
Given a chain of kernels $I_\circ$, find upper bounds on $n_0$ such that 
$$\left< \S_{n} I_{n_0} \right>_{\K[X_{[k] \times [n]}]} = I_n \ \mbox{  for all } n > n_0$$
 in terms of the ring homomorphisms $\phi_r$.  Of especial interest is when each $I_r$ is a toric ideal, in which case $\phi_r = \phi_A$ for a integral matrix $A$.
\end{ques}

In Section \ref{examples}, we showed that $\Pi$-invariant divisible semigroup rings $\K[Q]$  that are subrings of  $\K[X_{[k] \times \P}]$ are Noetherian $\K[Q] \ast \Pi$ modules.  A natural question is to what extent this property generalizes.

\begin{ques}\label{ques:gensemigroup}
Let $\K[Q] \subseteq \K[X_{[k] \times \P}]$ be a $\Pi$-invariant semigroup ring which is finitely generated under the action of $\Pi$.  Is it true that $\K[Q]$ is a Noetherian $\K[Q]\ast \Pi$-module?
\end{ques}

Alexei Krasilnikov constructed a remarkably simple example which shows that the answer to Question \ref{ques:gensemigroup} is ``no''.

\begin{ex}[Krasilnikov]\label{ex:krasilnikov}
Let $k = 2$ and let $\K[Q] \subseteq \K[X_{[2] \times \P}]$ be the semigroup generated by the monomials $x_{1i}x_{2j}$ where $i<j$.  Note that this semigroup ring is finitely generated up to the action of $\Pi$ by the single monomial $x_{11}x_{22}$.

For $n \geq 3$ define the element $w_{n} \in Q$ by
$$
w_{n} = x_{11}x_{2n} \prod_{i=1}^{n-1}x_{1i}x_{2i+1}.
$$
Consider the the multigrading on the ring $\K[Q]$ defined by $\deg x_{ij} = e_{j} \in \N^{\N}$.  In particular, the $\deg w_{n} = (2,2,\ldots, 2, 0, 0, \ldots)$.  Suppose that some $w_{m} |_{\Pi} w_{n}$.  Then there is a $ p \in \Pi$ such that $pw_{m} = h w_{n}$.  Now, $\deg p w_{m} \in \{0,2\}^{\N}$  so $\deg h \in   \{0,2\}^{\N}$ as well.  Examining at the right-most nonzero entry in $\deg h$, we see that $x_{2k_{1}}^{2} | h$ for some $k_{1}$.  Also, the right-most nonzero entry in $\deg pw_{m}$ implies that   $x_{2k_{2}}^{2} | p w_{m}$ for some $k_{2}$.  This implies that $x_{2k_{1}}^{2}x_{2k_{2}}^{2}$ divides $w_{n}$ which is impossible.
Hence,  the sequence $w_{3}, w_{4}, \ldots$ is a bad sequence and by Proposition \ref{prop:equivorder} and Theorem \ref{main_GB_thm}, $\K[Q]$ is not a Noetherian $\K[Q] \ast \Pi$-module. \qed
\end{ex}

While we have been mainly interested in ideals that are invariant under the action of the symmetric group, we needed to restrict to actions of the monoid of increasing functions $\Pi$ in order to prove our finiteness theorems.  We are lead to wonder if this strategy could always be used to prove Noetherianity under symmetric group actions or if there might be some pathological counterexamples or obstructions.

In particular, let $R$ be a ring equipped with an $\S_\P$ action.  We say that this action is $\S_\P$-finite if for every $f \in R$ there is an $m \in \P$ such that $\sigma \cdot f = \sigma' \cdot f$ for all $\sigma, \sigma' \in \S_\P$ such that $\sigma(i) = \sigma'(i)$ for all $i \leq m$.  If $R$ has a $\S_\P$-finite action, it also has a natural action by the monoid of increasing functions $\Pi$.

\begin{ques}
Is there a ring $R$ with a $\S_\P$-finite action such that $R$ is a Noetherian $R \ast \S_\P$-module but not a Noetherian $R \ast \Pi$-module?
\end{ques}

One of the lessons we have learned about proving Noetherianity of $\K[X_{[k] \times \P}]$ as a $\K[X_{[k] \times \P}]\ast \S_\P$-module is that it is not possible to define Gr\"obner bases in this setting.  This suggests that an approach for computing with ideals that have a natural symmetry group using Gr\"obner bases might not work well if the entire symmetry group is used.  However, working with a semigroup that has a $P$-order might be natural and useful, and not require the bookkeeping of a full symmetry group.  This suggests the following implementation challenge.  

\begin{pr}
Develop and implement algorithms for computing with symmetric ideals by using monoids of transformations and $P$-orders.
\end{pr}

For some recent work along these lines, including an algorithm for computing with certain classes of invariant ideals, we refer the reader to \cite{BrouwerDraisma}.


\section*{Acknowledgments}

We thank Alexei Krasilnikov for providing us with references to the work of Cohen and for the use of Example \ref{ex:krasilnikov}.


\end{document}